\documentclass[11pt]{article}
\usepackage[utf8]{inputenc}

\usepackage[T1]{fontenc}

\usepackage{amsmath,amssymb,amsthm}
\usepackage{blindtext}
\usepackage{graphicx}
\usepackage[svgnames]{xcolor}
\usepackage[normalem]{ulem}

\numberwithin{equation}{section}

\usepackage{wrapfig}

\usepackage[all]{xy}

\usepackage{enumitem}
\usepackage{xcolor}
\definecolor{darkcandyapplered}{rgb}{0.64, 0.0, 0.0}
\definecolor{joliblanc}{RGB}{252, 251, 249}
\definecolor{bluegray}{rgb}{0.4, 0.6, 0.8}
\definecolor{ceruleanblue}{rgb}{0.16, 0.32, 0.75}
\definecolor{greydark}{RGB}{75, 71, 59}
\definecolor{oldlavender}{RGB}{75, 71, 59}

\colorlet{shadecolor}{white}

\usepackage{hyperref}
\hypersetup{colorlinks=true,linkcolor=darkcandyapplered,citecolor=oldlavender}

\usepackage{caption}

\def\defn#1{{\bf\itshape #1}}

\usepackage{thmtools}

\declaretheoremstyle[%
  spaceabove=6pt,%
  spacebelow=6pt,%
  headfont=\normalfont\itshape,%
  postheadspace=1em,%
  qed=\qedsymbol,%
  headpunct={}
]{mystyle}

\renewcommand{\qedsymbol}{\color{oldlavender}{$\blacksquare$}}

\usepackage{amsthm}

\newtheoremstyle{mytheoremstyle} 
    {\topsep}                    
    {\topsep}                    
    {\itshape}                   
    {}                           
    {\scshape}                   
    {.}                          
    {.5em}                       
    {}  

\theoremstyle{mytheoremstyle}

\newtheorem{theorem}{Theorem}[section]
\newtheorem{proposition}[theorem]{Proposition}

\newtheorem{lemma}[theorem]{Lemma}

\newtheoremstyle{mydefinitionstyle} 
    {\topsep}                    
    {\topsep}                    
    {}                   
    {}                           
    {\itshape}                   
    {.}                          
    {.5em}                       
    {}  

\theoremstyle{mydefinitionstyle}

\newtheorem{definition}[theorem]{Definition}
\newtheorem*{remark}{Remark}
\newtheorem{example}[theorem]{Example}

\usepackage{geometry}
\usepackage{titlesec}
\usepackage{secdot}
\titleformat*{\section}{\centering\bfseries}
\titleformat*{\subsection}{\itshape}

\usepackage{enumitem}

\usepackage{bbm}

\newcommand{\D}{\mbox{D}}
\newcommand{\partialt}{\left. \frac{\mbox{d}}{\mbox{dt}}\right| _{t=0}}

\newcommand{\R}{\mathbb{R}}

\newcommand{\ad}{\mbox{ad}}
\newcommand{\Ad}{\mbox{Ad}}
\newcommand{\SO}{SO}

\newcommand{\TG}{T}
\renewcommand{\TH}{\widetilde{T}}
\renewcommand{\NG}{N}
\newcommand{\NH}{\widetilde{N}}
\renewcommand{\SS}{\mathfrak{s}}
\newcommand{\mH}{\mathfrak{p}}
\newcommand{\nH}{\widetilde{\mathfrak{n}}}

\newcommand\restr[2]{{\left.\kern-\nulldelimiterspace #1\vphantom{\big|} \right|_{#2}}}

\def\defn#1{{\itshape #1}}

\numberwithin{equation}{section}

\usepackage{authblk}
\title{Symplectic slice for subgroup actions}
\author{Marine Fontaine}

\date{}

\begin{document}
\maketitle

\begin{abstract}
Given a symplectic manifold endowed with a proper Hamiltonian action of a Lie group, we consider the action induced by a Lie subgroup. We propose a construction for two compatible Witt-Artin decompositions of the tangent space of the manifold, one relative to the action of the big group and one relative to the action of the subgroup. In particular, we provide an explicit relation between the respective symplectic slices.
\vspace{6pt}

\noindent \small\textit{Keywords:} Symplectic slice,  Hamiltonian dynamics, MGS normal form.

\end{abstract}

\section{Introduction}\label{Introduction}

Stability properties and bifurcations of relative equilibria is determined by a method developed by \cite{Krupa}, which states that the dynamics of an equivariant vector field in a neighbourhood of a group orbit is entirely governed by its transverse dynamics. This result uses the so-called \defn{slice coordinates} introduced by \cite{Palais}. Krupa first proves this result for compact Lie group actions and \cite{Fiedler} extend it to proper Lie group actions. The Hamiltonian analogue is studied by \cite{Mielke} and \cite{Guillemin}, as well as \cite{RobertsdeSousa} and \cite{RobertsWulffLamb}. By ``transverse dynamics'' we mean that the vector field in question can be split into two parts: one part is defined along the tangent space of the group orbit and the other part belongs to a choice of normal subspace. In Hamiltonian systems the flow of a Hamiltonian vector field with a fixed initial condition is confined to a level set of the momentum map, reflecting the conservation of momentum. The choice of normal subspace is therefore more restrictive than for general dynamical systems. Before giving its explicit form we introduce some terminologies: given a symplectic manifold $(M,\omega)$ acted on by a Lie group $G$, the action is called \defn{canonical} if it is smooth and it preserves the symplectic form $\omega$. For any element $x\in\mathfrak{g}$ of the Lie algebra $\mathfrak{g}$ of $G$ we denote by $x_M$ the vector field on $M$ generated by the action. A canonical action is \defn{Hamiltonian} if there exists a momentum map $\Phi_G:M\to\mathfrak{g}^*$ defined by the relation $\iota_{x_M}\omega=d\langle \Phi_G(\cdot),x\rangle$ for every $x\in\mathfrak{g}$. Here $\mathfrak{g}^*$ denotes the dual of the Lie algebra $\mathfrak{g}$.

\begin{definition}
	Assume that $G$ acts properly and canonically on a symplectic manifold $(M,\omega)$. In addition we assume that the action is Hamiltonian and that the associated momentum map $\Phi_G:M\to \mathfrak{g}^*$ is equivariant with respect to the action of $G$ on $M$ and the coadjoint action of $G$ on $\mathfrak{g}^*$. The quadruple $(M,\omega,G,\Phi_G)$ is called a \defn{Hamiltonian $G$-manifold}. 
\end{definition}

We fix a Hamiltonian $G$-manifold $(M,\omega,G,\Phi_G)$ and a point $m\in M$ with momentum $\mu=\Phi_G(m)$. The corresponding stabilizers for the action of $G$ on $M$ and the coadjoint action of $G$ on $\mathfrak{g}^*$ are denoted $G_m$ and $G_{\mu}$ respectively. Their respective Lie algebras are denoted $\mathfrak{g}_m$ and $\mathfrak{g}_{\mu}$. Let $G\cdot m=\lbrace g\cdot m\mid g\in G\rbrace$ be the \defn{$G$-orbit} of $m\in M$ and denote by $\mathfrak{g}\cdot m$ its tangent space at $m$. Elements of $\mathfrak{g}\cdot m$ are vectors of the form $x_M(m):=\partialt \mbox{exp}(tx)\cdot m$, where $x\in\mathfrak{g}$ and $\mbox{exp}:\mathfrak{g}\to G$ is the group exponential. A \defn{symplectic slice} $\NG_1$ at $m$ is a $G_m$-invariant subspace of $(T_mM,\omega(m))$ defined by

\begin{equation}\label{NG_1}
		\NG_1:=\ker(\D\Phi_G(m))/\mathfrak{g}_{\mu}\cdot m,
\end{equation}
where $\D\Phi_G(m)$ is the differential of $\Phi_G$ at $m$. It is endowed with a symplectic structure $\omega_{\NG_1}$ coming from $\omega(m)$, and a linear Hamiltonian action of $G_m$ that makes it a Hamiltonian $G_m$-space.  This subspace is of particular interest for the study of stability, persistence and bifurcations of relative equilibria (cf. \cite{PatrickRoberts,LermanSinger,OrtegaRatiu}, as well as \cite{MontaldiRodriguez}). The construction of a symplectic slice is based on a \defn{Witt-Artin decomposition} (relative to the $G$-action) of $T_mM$ i.e a decomposition into four $G_m$-invariant subspaces:

\begin{equation}\label{intro: witt G}
		T_mM=\TG_0\oplus\TG_1\oplus\NG_0\oplus\NG_1
\end{equation}
with respect to which the skew-symmetric matrix associated to $\omega(m)$ has a specific normal form. The part $\TG_0\oplus \TG_1=\mathfrak{g}\cdot m$ corresponds to the directions tangent to $G\cdot m$ whereas the part $\NG_0\oplus \NG_1$ is a choice of normal subspaces such that $\TG_0\oplus \NG_1=\ker\left(\D\Phi_G(m)\right)$. This decomposition first appears in \cite{witt} for symmetric bilinear forms. The construction involves choices, namely the subspaces $\TG_1,\NG_0$ and $\NG_1$.

Let $H\subset G$ be a Lie subgroup with Lie algebra $\mathfrak{h}$ and inclusion map $i_{\mathfrak{h}}:\mathfrak{h}\hookrightarrow \mathfrak{g}$. The dual map $\mathfrak{i}_{\mathfrak{h}}^*:\mathfrak{g}^*\to \mathfrak{h}^*$ is given by $i_{\mathfrak{h}}^*(\lambda)=\restr{\lambda}{\mathfrak{h}}$ which is the restriction of the linear form $\lambda$ to the subalgebra $\mathfrak{h}$. Note that by definition, the projection $i_{\mathfrak{h}}^*:\mathfrak{g}^*\to \mathfrak{h}^*$ is $H$-equivariant. As the action of $H$ on $M$ is still Hamiltonian, it admits a momentum map $\Phi_H:M\to \mathfrak{h}^*$ given by $\Phi_H=i_{\mathfrak{h}}^*\circ\Phi_G$. Therefore $(M,\omega,H,\Phi_H)$ is a Hamiltonian $H$-manifold and we call $\Phi_H$ the \defn{induced momentum map}. In this case, we can also consider a Witt-Artin decomposition of $T_mM$ relative to the $H$-action:

\begin{equation}\label{intro: witt H}
		T_mM=\TH_0\oplus\TH_1\oplus\NH_0\oplus \NH_1.
\end{equation}
In particular, the $H_m$-invariant subspace $\NH_1$ is a symplectic slice for the $H$-action. It is chosen such that

\begin{equation}
	\NH_1:=\ker(\D\Phi_H(m))/\mathfrak{h}_{\alpha}\cdot m,
\end{equation}
where $\alpha:=\Phi_H(m)$ is the restriction of the linear form $\mu\in\mathfrak{g}^*$ to $\mathfrak{h}$. In general two arbitrary decompositions \eqref{intro: witt G} and \eqref{intro: witt H} cannot be compared.

In the study of explicit symmetry breaking phenomenons, Hamiltonian equations are perturbed in a way that the symmetry group $G$ breaks into one of its subgroup $H$. This phenomenon is studied by many authors. References for the non-Hamiltonian case are for instance \cite{MR1248622} or \cite{MR2092068}. Some aspects of the Hamiltonian case are studied in \cite{MR879691}, \cite{Grasbi}, \cite{MR2684603} or \cite{thesis}. The stability properties of the perturbed system rely on a symplectic slice relative to the $H$-action on $M$, which is ``bigger'' than a slice relative to the $G$-action.  This leads us to find explicit relations between $\NG_1$ and $\NH_1$. It is implicitly used in \cite{Grasbi} that if $G$ is a torus and $H$ is a subtorus, both acting freely on $M$, a symplectic slice $\NH_1$ at $m$ can be chosen of the form

\begin{equation}\label{grasbi}
		\NH_1= \NG_1\oplus X_m
\end{equation}
for some subspace $X_m\subset T_mM$ isomorphic to $\mathfrak{g}/\mathfrak{h}\oplus (\mathfrak{g}/\mathfrak{h})^*$. We generalize this observation for non-abelian connected Lie groups and non-free actions but with the assumption:

\begin{equation}\label{A}\tag{A}
	G_m\mbox{ acts on }H\mbox{ by conjugation.}
\end{equation} 
This assumption is required if we want the subspaces \eqref{Hperpmu} and \eqref{m et n} below to be $G_m$-invariant. Given \eqref{A}, we construct a Witt-Artin decomposition \eqref{intro: witt G} at $m$ (relative to the $G$-action) and a Witt-Artin decomposition \eqref{intro: witt H} at $m$ (relative to the $H$-action) that are compatible in the sense that the symplectic slice $\NH_1$ for $H$ can be expressed in terms of the symplectic slice $\NG_1$ for $G$. Explicitly, $\NH_1$ can be chosen as:
\begin{equation}\label{intro:slice decomposition}
	\NH_1=\NG_1\oplus X_m\oplus \SS(G,H,\mu)\cdot m
\end{equation}
where $X_m$ is symplectomorphic to some canonical cotangent bundle $\mathfrak{b}\oplus\mathfrak{b}^*$, and $\mathfrak{s}(G,H,\mu)\cdot m$ is identified with a symplectic slice for the $H$-action on the coadjoint orbit $G\cdot \mu$ (cf. Proposition \ref{s as a slice}). In the case of free actions, $\mathfrak{b}\oplus\mathfrak{b}^*$ is some complement to $\mathfrak{h}_{\mu}\oplus\mathfrak{h}_{\mu}^*$ in $\mathfrak{g}_{\mu}\oplus\mathfrak{g}_{\mu}^*$. If in addition $G$ is abelian, we recover \eqref{grasbi} because $\mathfrak{s}(G,H,\mu)\cdot m$ is trivial.

The paper is organized as follows. In $\S$\ref{section 1} we explain how to make a specific choice of subspace $\TG_1$ in \eqref{intro: witt G} to find a compatible decomposition \eqref{intro: witt H}. In $\S$\ref{section 2}, the Symplectic Tube Theorem is recalled. This yields a specific normal form for the momentum map (cf. Theorem \ref{MGS}). In $\S$\ref{section 3} the construction introduced in $\S$\ref{section 1} is used to show that a symplectic slice $\NH_1$ can be chosen of the form \eqref{intro:slice decomposition} (cf. Theorem \ref{cor: KerJH}). The symplectic form and the momentum map on it are specified (cf. Theorem \ref{cor: KerH symplectic form} and Proposition \ref{mm splitting}). The last section discusses the construction of the other subspaces which appear in \eqref{intro: witt G} and \eqref{intro: witt H}.

\paragraph{Acknowledgements} I would like to thank James Montaldi and Miguel Rodriguez-Olmos for their comments on an early draft of the manuscript. I also thank the anonymous referees for their suggestions. This work forms a part of my Ph.D. thesis~\cite{thesis}.

\section{Witt-Artin decomposition}\label{section 1}

In this section we introduce a splitting of the Lie algebra $\mathfrak{g}$ in order to construct compatible Witt-Artin decompositions \eqref{intro: witt G} and \eqref{intro: witt H}. We first fix some notations: there are natural actions of $G$ on $\mathfrak{g}$ and $\mathfrak{g}^*$, namely the \defn{adjoint action} $\Ad:(g,x)\in G\times\mathfrak{g}\mapsto\Ad_gx\in\mathfrak{g}$ and the \defn{coadjoint action} $\Ad^*:(g,\lambda)\in G\times\mathfrak{g}^*\mapsto \Ad^*_{g^{-1}}\lambda\in\mathfrak{g}^*$. If $\langle\cdot,\cdot\rangle$ denotes the canonical pairing between $\mathfrak{g}^*$ and $\mathfrak{g}$ then $\langle \Ad^*_{g^{-1}}\lambda,x\rangle=\langle \lambda, \Ad_{g^{-1}}x\rangle$ for every $\lambda\in\mathfrak{g}^*$ and $x\in\mathfrak{g}$. The respective infinitesimal actions are given by
$\ad:(x,y)\in \mathfrak{g}\times\mathfrak{g}\mapsto\ad_xy=[x,y]\in\mathfrak{g}$ and $\ad^*:(x,\lambda)\in \mathfrak{g}\times\mathfrak{g}^*\mapsto\ad^*_x\lambda\in\mathfrak{g}^*$, where $\langle \ad_x^*\lambda,y\rangle=\langle \lambda,[x,y]\rangle$ for every $\lambda\in\mathfrak{g}^*$ and $x,y\in\mathfrak{g}$. Furthermore if $(V,\omega)$ is a symplectic vector space and $W\subset V$ is a subspace, the \defn{symplectic orthogonal} $W^{\omega}$ of $W$ in $V$ is the set of vectors $v\in V$ such that $\omega(v,w)=0$ for every $w\in W$. If $W,U$ are two subspaces such that $U\subset W\subset V$, then $U^{\perp_W}$ denotes a \defn{complement} of $U$ in $W$ so that $U\oplus U^{\perp_W}=W$ is a direct sum.

Let $m\in M$ with momentum $\mu=\Phi_G(m)$ and assume that \eqref{A} is satisfied. In particular $G_m$ acts on the stabilizer subalgebras $\mathfrak{h}_m$ and $\mathfrak{h}_{\mu}$ by mean of the adjoint action. We split the Lie algebra $\mathfrak{g}$ into three parts

\begin{equation}\label{splitting of Lie algebra g}
	\mathfrak{g}=\mathfrak{g}_m\oplus \mathfrak{m}\oplus \mathfrak{n}.
\end{equation}
for some $G_m$-invariant subspaces $\mathfrak{m}$ and $\mathfrak{n}$, chosen as described below. Since the $G$-action on $M$ is proper, the stabilizer $G_m$ is compact. The Lie subalgebra $\mathfrak{g}_m$ can thus be decomposed into a direct sum of $G_m$-invariant subspaces $\mathfrak{g}_m=\mathfrak{h}_m\oplus \mathfrak{h}_m^{\perp_{\mathfrak{g}_m}}$. Note that $\mathfrak{h}_m$ is $G_m$-invariant by assumption \eqref{A}. Similarly $\mathfrak{h}_{\mu}=\mathfrak{h}_m \oplus \mH$ for some $G_m$-invariant complement $\mH$. Then $$\mathfrak{g}_m+\mathfrak{h}_{\mu}=\mathfrak{h}_m\oplus \mathfrak{h}_m^{\perp_{\mathfrak{g}_m}}\oplus \mH.$$ Since $\mathfrak{g}_m+\mathfrak{h}_{\mu}\subset \mathfrak{g}_{\mu}$, we can choose a $G_m$-invariant complement $\mathfrak{b}:=(\mathfrak{g}_m+\mathfrak{h}_{\mu})^{\perp_{\mathfrak{g}_{\mu}}}$ so that

\begin{equation}\label{gmu decomposition}
		\mathfrak{g}_{\mu}=\underbrace{\mathfrak{h}_m\oplus \mH}_{\mathfrak{h}_{\mu}} \oplus  \mathfrak{h}_m^{\perp_{\mathfrak{g}_m}}\oplus \mathfrak{b}=\underbrace{\mathfrak{h}_m\oplus \mathfrak{h}_m^{\perp_{\mathfrak{g}_m}}}_{\mathfrak{g}_{m}} \oplus  \mH \oplus \mathfrak{b}.
\end{equation}
In particular the $G_m$-invariant subspace $\mathfrak{m}$ in \eqref{splitting of Lie algebra g} is chosen of the form $\mathfrak{m}:=\mH\oplus \mathfrak{b}$. By \eqref{gmu decomposition} it satisfies $\mathfrak{g}_{\mu}=\mathfrak{g}_m\oplus\mathfrak{m}$.

To define the subspace $\mathfrak{n}$ in \eqref{splitting of Lie algebra g} we introduce the ``symplectic orthogonal'':

\begin{equation}\label{Hperpmu}
	\mathfrak{h}^{\perp_{\mu}}:=\left\lbrace x\in\mathfrak{g}\mid x_M(m)\in (\mathfrak{h}\cdot m)^{\omega(m)}\right\rbrace.
\end{equation}
This subspace is present in the context of geometric quantization (cf. \cite{Duval}) and is $G_m$-invariant by assumption \eqref{A}. It is characterized as follows:

\begin{proposition}\label{prop: kmu}
		The conditions below are equivalent:
		
	\begin{enumerate}[label=(\roman*)]
		\item $x\in \mathfrak{h}^{\perp_{\mu}}$.
		\item $\langle \mu, [x,\eta]\rangle=0$ for every $\eta\in\mathfrak{h}$.
		\item $\ad_x^*\mu\in\mathfrak{h}^{\circ}$ where $\mathfrak{h}^{\circ}:=\lbrace \lambda\in\mathfrak{g}^*\mid \restr{\lambda}{\mathfrak{h}}=0\rbrace$ is the annihilator of $\mathfrak{h}$ in $\mathfrak{g}^*$.
	\end{enumerate}
\end{proposition}

\begin{proof}
Let $m\in M$ and $\mu=\Phi_G(m)$. We first show that $(i)\iff (ii)$.

	\begin{eqnarray*}
		x\in\mathfrak{h}^{\perp_{\mu}} & \iff & x_M(m)\in(\mathfrak{h}\cdot m)^{\omega(m)}\\
		& \iff &\omega(m)(x_M(m),\eta_M(m))=0\mbox{ for all }\eta\in\mathfrak{h}\\
		& \iff &\langle \mu,[x,\eta]\rangle=0\mbox{ for all }\eta\in\mathfrak{h}.
	\end{eqnarray*}
	
The last step above follows from the calculation: 

	\begin{eqnarray*}
		\omega(m)(x_M(m),\eta_M(m))& = &\langle D\Phi_G(m)\cdot \eta_M(m),x\rangle\\
		& = &\partialt \langle \Phi_G(\mbox{exp}(t\eta)\cdot m),x\rangle\\
		& = &\partialt \langle \Ad^*_{\mbox{exp}(-t\eta)}\Phi_G(m),x\rangle\\
		& = &\langle \Phi_G(m),\partialt \Ad_{\mbox{exp}(-t\eta)}x\rangle\\
		& = &\langle \mu,[x,\eta]\rangle.
	\end{eqnarray*}
	
Finally, $(ii)\iff (iii)$ since

	\begin{eqnarray*}
		\langle \mu,[x,\eta]\rangle=0\mbox{ for all }\eta\in\mathfrak{h} & \iff &\langle \ad^*_x\mu,\eta\rangle=0\mbox{ for all }\eta\in\mathfrak{h}\\
& \iff &\ad^*_x\mu\in\mathfrak{h}^{\circ}.
	\end{eqnarray*}
\end{proof}

Consider the projection $\alpha:=\restr{\mu}{\mathfrak{h}}\in\mathfrak{h}^*$ and observe that $\mathfrak{h}_{\alpha}\subset \mathfrak{h}^{\perp_{\mu}}$. Indeed let $x\in \mathfrak{h}_{\alpha}$, $\eta\in\mathfrak{h}$, and note that

\begin{equation}
	\langle \ad_x^*\mu,\eta\rangle=\langle \mu,[x,\eta]\rangle=\langle \alpha,[x,\eta]\rangle=\langle \ad_x^*\alpha,\eta\rangle=0.
\end{equation}
 Furthermore $\mathfrak{g}_{\mu}\subset\mathfrak{h}^{\perp_{\mu}}$ since $\mathfrak{g}_{\mu}=\lbrace x\in\mathfrak{g}\mid \ad^*_x\mu=0\rbrace$. We conclude that the inclusion $\mathfrak{g}_{\mu}+\mathfrak{h}_{\alpha}\subset\mathfrak{h}^{\perp_{\mu}}$ holds.  Using \eqref{gmu decomposition} and $\mathfrak{g}_{\mu}\cap \mathfrak{h}_{\alpha}=\mathfrak{h}_{\mu}$, we choose a $G_m$-invariant complement $\mathfrak{a}$ such that
 
\begin{equation}\label{gmu+Ha}
		\mathfrak{g}_{\mu}+\mathfrak{h}_{\alpha}=\underbrace{\mathfrak{h}_{\mu}\oplus \mathfrak{h}_m^{\perp_{\mathfrak{g}_m}}\oplus\mathfrak{b}}_{\mathfrak{g}_{\mu}}\oplus\mathfrak{a}=\underbrace{\mathfrak{h}_{\mu}\oplus\mathfrak{a}}_{\mathfrak{h}_{\alpha}}\oplus \mathfrak{h}_m^{\perp_{\mathfrak{g}_m}}\oplus\mathfrak{b}.
\end{equation}
Choose $\SS(G,H,\mu)$ to be a $G_m$-invariant complement to $\mathfrak{g}_{\mu}+\mathfrak{h}_{\alpha}$ in $\mathfrak{h}^{\perp_{\mu}}$. We can thus express  \eqref{Hperpmu} as a direct sum of $G_m$-invariant subspaces

\begin{equation}\label{Hperpmu new}
	\mathfrak{h}^{\perp_{\mu}}=\mathfrak{g}_{\mu}\oplus\mathfrak{a}\oplus \SS(G,H,\mu).
\end{equation}
In particular,

\begin{equation}\label{new q}
	\mathfrak{q}:=\mathfrak{a}\oplus \SS(G,H,\mu)
\end{equation} 
is a $G_m$-invariant complement to $\mathfrak{g}_{\mu}$ in $\mathfrak{h}^{\perp_{\mu}}$.
Finally, choosing a $G_m$-invariant complement $(\mathfrak{h}^{\perp_{\mu}})^{\perp_{\mathfrak{g}}}$ of $\mathfrak{h}^{\perp_{\mu}}$ in $\mathfrak{g}$ yields the decomposition

\begin{equation}\label{g decomposition1}
	\mathfrak{g}=\underbrace{\mathfrak{h}_{\mu}\oplus \mathfrak{h}_m^{\perp_{\mathfrak{g}_m}}\oplus\mathfrak{b}}_{\mathfrak{g}_{\mu}}\oplus \underbrace{\mathfrak{q}\oplus (\mathfrak{h}^{\perp_{\mu}})^{\perp_{\mathfrak{g}}}}_{\mathfrak{n}}.
\end{equation}
By \eqref{gmu decomposition} and \eqref{g decomposition1}, the $G_m$-invariant subspaces $\mathfrak{m}$ and $\mathfrak{n}$ of \eqref{splitting of Lie algebra g} are

\begin{equation}\label{m et n}
	\mathfrak{m}:=\mH \oplus \mathfrak{b}\quad\mbox{ and }\quad\mathfrak{n}=\mathfrak{q}\oplus (\mathfrak{h}^{\perp_{\mu}})^{\perp_{\mathfrak{g}}}.
\end{equation}

\begin{remark}
	The subspace $\mathfrak{b}\subset \mathfrak{g}_{\mu}$ is not a Lie subalgebra in general. However if $G$ is compact and $\mu$ is a regular value of $\Phi_G$, then $G_{\mu}$ is a maximal torus of $G$. In this case $K=G_mH_{\mu}$ is a subgroup of $G_{\mu}$ since $G_mH_{\mu}=H_{\mu}G_m$. It is a Lie subgroup by closedness of $G_m$ and $H_{\mu}$ and its Lie algebra is $\mathfrak{k}=\mathfrak{g}_m+\mathfrak{h}_{\mu}$. Since $\mathfrak{g}_{\mu}$ is abelian $\mathfrak{k}$ is trivially an ideal of $\mathfrak{g}_{\mu}$ making $\mathfrak{b}$ isomorphic to the Lie algebra $\mathfrak{g}_{\mu}/\mathfrak{k}$. If in addition $H_m=\mathbbm{1}$, then $K=G_m\rtimes H_{\mu}$, as $G_m$ is normal in $K$.
\end{remark}

The next theorem is a standard result. A proof can be found for example in \cite{MR2021152} or \cite{MR3242761}.

\begin{theorem}[Witt-Artin decomposition]\label{Witt}
Let $(M,\omega,G,\Phi_G)$ be a Hamiltonian $G$-manifold and let $m\in M$ with momentum $\mu=\Phi_G(m)$. Fix a splitting of $\mathfrak{g}$ into $G_m$-invariant subspaces as in \eqref{splitting of Lie algebra g}. Then the tangent space $T_mM$ at $m$ decomposes as:

	\begin{equation}
		T_mM=\TG_0\oplus \TG_1\oplus \NG_0\oplus \NG_1,
	\end{equation}
where the subspaces $\TG_0, \TG_1,\NG_0,\NG_1$ are constructed as follows:

	\begin{enumerate}[label=(\roman*)]
		\item $\TG_0:=\ker\left(\D\Phi_G(m)\right)\cap \mathfrak{g}\cdot m=\mathfrak{g}_{\mu}\cdot m$.
				
		\item $\TG_1:=\mathfrak{n}\cdot m$ which is a symplectic vector subspace of $(T_mM,\omega(m))$.\label{TG1}
		
		\item $\NG_1$ is a choice of $G_m$-invariant complement to $\TG_0$ in $\ker\left(\D\Phi_G(m)\right)$. It is a symplectic subspace of $(T_mM,\omega(m))$ with symplectic form $\omega_{\NG_1}:=\restr{\omega(m)}{\NG_1}$. This subspace is called a \defn{symplectic slice}. The linear action of $G_m$ on $\NG_1$ is Hamiltonian with momentum map $\Phi_{\NG_1}:\NG_1\to\mathfrak{g}_m^*$ given by $\langle\Phi_{\NG_1}(\nu),x\rangle=\frac{1}{2}\omega(m)(x_{\NG_1}(\nu),\nu)$ for $\nu\in\NG_1$ and $x\in\mathfrak{g}_m$. The infinitesimal generator $x_{\NG_1}(\nu)=\D x_M(m)\cdot \nu\in \NG_1$ is the linearisation of the vector field $x_M$ at $m$ (since $x\in\mathfrak{g}_m$, $x_M(m)=0$).
			
	\item $\NG_0$ is a $G_m$-invariant Lagrangian complement to $\TG_0$ in the symplectic orthogonal $(\TG_1\oplus \NG_1)^{\omega(m)}$. There is an isomorphism $f:\NG_0\to \mathfrak{m}^*$ given by $\langle f(w),y\rangle=\omega(m)\left(y_M(m),w\right)$ for every $w\in \NG_0$ and $y\in \mathfrak{m}$.\label{f definition}
	\end{enumerate}
	
Furthermore, the subspaces $\TG_1,\NG_1$ and $\TG_0\oplus\NG_0$ are mutually symplectically orthogonal.
\end{theorem}

\begin{remark}\label{symplectic form on T0+N0}
	\normalfont In \ref{TG1} of Theorem \ref{Witt}, the symplectic form $\omega(m)$ restricted to $\TG_1$ coincides with the Kostant-Kirillov-Souriau symplectic form. Besides the symplectic form $\omega(m)$ restricted to $\TG_0\oplus\NG_0$ takes the form
	
	\begin{equation*}
		\omega(m)(x_M(m)+w,x'_M(m)+w')=\langle f(w'),x\rangle-\langle f(w),x'\rangle
	\end{equation*} 
	
	for $x,x'\in\mathfrak{m}, w,w'\in\NG_0$ and $f:\NG_0\to\mathfrak{m}^*$ as in \ref{f definition}. Indeed since $y_M(m)=0$ for all $y\in \mathfrak{g}_m$, the elements of $\TG_0$ are of the form $x_M(m)$ with $x\in\mathfrak{m}$. Let $x,x'\in\mathfrak{m}$ and $w,w'\in\NG_0$. As both $\TG_0$ and $\NG_0$ are Lagrangian subspaces of $\TG_0\oplus\NG_0$, $$\omega(m)(x_M(m)+w,x'_M(m)+w')=\omega(m)(x_M(m),w')+\omega(w,x'_M(m))$$ which is equal to $\langle f(w'),x\rangle-\langle f(w),x'\rangle$ by definition of $f$.
\end{remark}

\vspace{0.5cm}
Applying Theorem \ref{Witt} to $(M,\omega,G,\Phi_G)$ with the subspaces in \eqref{splitting of Lie algebra g} taken as in \eqref{m et n}, we get a decomposition 

	\begin{equation}\label{Witt Artin G}
		T_mM=\TG_0\oplus \TG_1\oplus \NG_0\oplus \NG_1
	\end{equation}
with $\TG_0=(\mathfrak{g}_m\oplus\mathfrak{p}\oplus\mathfrak{b})\cdot m$ and $\TG_1=(\mathfrak{q}\oplus(\mathfrak{h}^{\perp_{\mu}})^{\perp_{\mathfrak{g}}})\cdot m$. Note that we have some freedom in the choice of $G_m$-invariant normal subspaces $\NG_0$ and $\NG_1$. As we did previously we set $\alpha:=\Phi_H(m)=\restr{\mu}{\mathfrak{h}}$ and we define $\TH_0=\mathfrak{h}_{\alpha}\cdot m$. We shall give a specific choice of subspaces $\TH_1,\NH_0,\NH_1$ such that the tangent space of $M$ at $m$ decomposes as

	\begin{equation}\label{Witt Artin H}
		T_mM=\TH_0\oplus \TH_1\oplus \NH_0\oplus \NH_1
	\end{equation}
which is compatible with both, the decomposition \eqref{Witt Artin G} and the construction of Theorem \ref{Witt} applied to $(M,\omega, H,\Phi_H)$.

\section{Symplectic Tube Theorem}\label{section 2}

In this section we introduce a fundamental result to study both the local dynamics and the local geometry of Hamiltonian $G$-manifolds. It provides a local model for a $G$-invariant open neighbourhood of a $G$-orbit in $M$. Take a point $m\in M$ with momentum $\mu=\Phi_G(m)$ and choose a splitting as in \eqref{splitting of Lie algebra g}. Let $\NG_1$ be a symplectic slice at $m$. Since $\NG_1$ is $G_m$-invariant, there is an action of $G_m$ on the product $G\times \mathfrak{m}^*\times \NG_1$ given by 

	\begin{equation}
		k\cdot (g,\rho,\nu)=(gk^{-1},\Ad^*_{k^{-1}}\rho,k\cdot\nu).
	\end{equation}
This action is free and proper by freeness and properness of the action on the $G$-factor. In particular the orbit space $Y$ is a smooth manifold whose points are equivalence classes of the form $[g,\rho,\nu]$. We denote the orbit map by $\pi:G\times \mathfrak{m}^*\times\NG_1\to Y$. The group $G$ acts smoothly and properly on $Y$, by left multiplication on the $G$-factor. Let $\mathfrak{m}^*_0\subset \mathfrak{m}^*$ and $(\NG_1)_0\subset \NG_1$ be $G_m$-invariant neighbourhoods of zero in $\mathfrak{m}^*$ and $\NG_1$, respectively. Then

	\begin{equation}\label{T_G}
		Y_0:=G\times_{G_m}(\mathfrak{m}^*_0\times (\NG_1)_0)
	\end{equation}
is a neighbourhood of the zero section in $Y$. Given two elements $V_i=T_{(g,\rho,\nu)}\pi\cdot\left(T_eL_g\cdot\xi_i,\dot{\rho}_i,\dot{\nu}_i\right)\in T_{[g,\rho,\nu]}Y_0$ for $i=1,2$, define the closed $G$-invariant $2$-form

	\begin{eqnarray*}
		\omega_{Y_0}([g,\rho,\nu])\left(V_1,V_2\right)
		& = &\langle\dot{\rho}_2+\D\Phi_{\NG_1}(\nu)\cdot\dot{\nu}_2,\xi_1\rangle-\langle\dot{\rho}_1+\D\Phi_{\NG_1}(\nu)\cdot\dot{\nu}_1,\xi_2\rangle\\
		&&+\langle\rho+\Phi_{\NG_1}(\nu),[\xi_1,\xi_2]\rangle+\omega(m)\left((\xi_1)_M(m),(\xi_2)_M(m)\right)\\
		&&+\omega_{\NG_1}(\dot{\nu}_1,\dot{\nu}_2).\\
	\end{eqnarray*}

There is a neighbourhood $Y_0\subset Y$ as above such that the $2$-form $\omega_{Y_0}$ is non-degenerate, turning $(Y_0,\omega_{Y_0})$ into a symplectic manifold (cf. \cite{MR2021152} Proposition $7.2.2$).

Let $Z^2(\mathfrak{g})$ be the space of closed $2$-forms on $\mathfrak{g}$. Define the \defn{Chu map} $\Psi: M\to Z^2(\mathfrak{g})$ associated to the $G$-action by 

	\begin{equation}
		\Psi(m)(x,y):=\omega(m)(x_M(m),y_M(m)).
	\end{equation}
In the proof of Proposition \ref{prop: kmu} we calculated $\Psi(m)(x,y)=\langle \mu,[x,y]\rangle$. Hence $\Psi(m)$ coincides with the Kostant-Kirillov-Souriau symplectic form on the coadjoint orbit $G\cdot \mu$, whenever $x,y\in \mathfrak{n}$.

\vspace{0.5cm}
The next theorem is the well-known Symplectic Tube Theorem. It was obtained by \cite{Marle85} and generalized by \cite{Guillemin,MR1486529,MR2021152}. A proof is available in~\cite{MR2021152} (cf. Theorem $7.4.1$).

\begin{theorem}[Symplectic Tube Theorem]\label{symplectic tube thm}
	Let $(M,\omega,G,\Phi_G)$ be a Hamiltonian $G$-manifold. Let $m\in M$ with momentum $\mu=\Phi_G(m)$. Given $(Y_0,\omega_{Y_0})$ as above, there exists a $G$-invariant neighbourhood $U\subset M$ of $m$ and a $G$-equivariant symplectomorphism $$\varphi:(Y_0,\omega_{Y_0})\to (U,\restr{\omega}{U})$$ such that $\varphi([e,0,0])=m$.
\end{theorem}

We call the triplet $(\varphi,Y_0,U)$ a \defn{symplectic $G$-tube} at $m$ and we also say that $(Y_0,\omega_{Y_0})$ is a \defn{symplectic local model} for $(U,\restr{\omega}{U})$. Besides the momentum map $\Phi_G:M\to\mathfrak{g}^*$ can be expressed in terms of the slice coordinates:

\begin{theorem}[Marle-Guillemin-Sternberg Normal Form]\label{MGS}
	Let $(M,\omega,G,\Phi_G)$ be a Hamiltonian $G$-manifold and let $(\varphi,Y_0,U)$ be a symplectic $G$-tube at $m\in M$. Then the $G$-action on $Y_0$ is Hamiltonian with associated momentum map $\widetilde{\Phi}_G:Y_0\to \mathfrak{g}^*$ defined by
	
	\begin{equation}\label{Jyr}
		\widetilde{\Phi}_G([g,\rho,\nu])=\Ad^*_{g^{-1}}(\Phi_G(m)+\rho+\Phi_{N_1}(\nu)).
	\end{equation}
If $G$ is connected, $\widetilde{\Phi}_G$ coincides with $\restr{\Phi_G}{U}$ when pulled back along $\varphi^{-1}$.
\end{theorem}

\section{Compatible symplectic slices}\label{section 3}

In this section we explain how to choose the symplectic slice $(\NH_1,\omega_{\NH_1})$ at $m$ arising in \eqref{Witt Artin H}. Explicitly

	\begin{equation}
		\NH_1= \mathfrak{s}(G,H,\mu)\cdot m\oplus X_m\oplus \NG_1,
	\end{equation}
where $\SS(G,H,\mu)$ is a $G_m$-invariant complement to $\mathfrak{g}_{\mu}+\mathfrak{h}_{\alpha}$ in $\mathfrak{h}^{\perp_{\mu}}$ (cf. \eqref{gmu+Ha} and \eqref{Hperpmu new}), and $X_m\subset T_mM$ is some subspace symplectomorphic to $\mathfrak{b}\oplus\mathfrak{b}^*$ with the canonical symplectic form. We show in Lemma \ref{symplectic thing st} that $\mathfrak{s}(G,H,\mu)\cdot m$ is a symplectic subspace of $(T_mM,\omega(m))$. The next proposition provides a geometric description of $\mathfrak{s}(G,H,\mu)\cdot m$ as it appears in \cite{Perlmutter}.

\begin{proposition}\label{s as a slice}
	The subspace $\mathfrak{s}(G,H,\mu)\cdot m$ is identified with a symplectic slice at $\mu$ for the $H$-action on the coadjoint orbit $G\cdot\mu$.
\end{proposition}

\begin{proof}
	The subgroup $H$ acts on the coadjoint orbit $G\cdot \mu$ by left multiplication. Since the momentum map for the standard $G$-action on $G\cdot \mu$ is just the inclusion $G\cdot \mu\hookrightarrow\mathfrak{g}^*$, the momentum map $\Phi:G\cdot\mu\to \mathfrak{h}^*$ for the $H$-action is given by $\Phi(\Ad^*_{g^{-1}} \mu)=i^*_{\mathfrak{h}}(\Ad^*_{g^{-1}} \mu)$. The kernel of its differential is $\ker(\D\Phi(\mu))=(\mathfrak{a}\oplus\mathfrak{s}(G,H,\mu))\cdot\mu$. Indeed, denoting by $x_{\mathfrak{g}^*}(\mu)=-\ad_x^*\mu$ an element of $T_{\mu}(G\cdot\mu)$, a straightforward calculation shows that 
	
	\begin{equation*}
		x_{\mathfrak{g}^*}(\mu)\in \ker\left(\D\Phi(\mu)\right)\quad\Longleftrightarrow\quad -\ad_x^*\mu\in\mathfrak{h}^{\circ}.
	\end{equation*}
By Proposition \ref{prop: kmu}, $x\in \mathfrak{h}^{\perp_{\mu}}$. By using the identification $\mathfrak{g}^*=\mathfrak{n}^{\circ}\oplus T_{\mu}(G\cdot\mu)$ and \eqref{Hperpmu new}, $x\in \mathfrak{a}\oplus\mathfrak{s}(G,H,\mu)$. The momentum of $\mu$ is $\Phi(\mu)=i^*_{\mathfrak{h}}(\mu)=\alpha$. Hence a symplectic slice for the $H$-action on $G\cdot\mu$ is a complement to $\mathfrak{h}_{\alpha}\cdot\mu$ in $\ker(\D\Phi(\mu))$. By construction, this complement is $\mathfrak{s}(G,H,\mu)\cdot \mu$ which can be identified with $\mathfrak{s}(G,H,\mu)\cdot m$ since $\mathfrak{s}(G,H,\mu)$ has trivial intersection with $\mathfrak{g}_m$ and $\mathfrak{g}_{\mu}$.
\end{proof}

\begin{proposition}\label{KerJH}
Let $(M,\omega,G,\Phi_G)$ be a Hamiltonian $G$-manifold and let $\Phi_H:M\to \mathfrak{h}^*$ be the induced momentum map. Then $$\ker\left(\D \Phi_H(m)\right)= \ker\left(\D \Phi_G(m)\right)\oplus \mathcal{M},$$ where $\mathcal{M}\subset T_mM$ is isomorphic to $\mathfrak{q}\cdot m \oplus \mathfrak{b}^*$ as defined in \eqref{m et n}. 
\end{proposition}

\begin{remark}
Note that we do not need the assumption of Theorem \ref{MGS} that $G$ is connected because the statement only depends on the differential.
\end{remark}

\begin{proof}
	It is clear from the definitions that there is an inclusion of subspaces

	\begin{equation}\label{kerg in kerk}
		\ker\left(D\Phi_G(m)\right)\subset \ker\left(D\Phi_H(m)\right).
	\end{equation} 
Let $\left(\varphi, G\times_{G_m} \left(\mathfrak{m}^*_0\times (\NG_1)_0\right),U\right)$ be a symplectic $G$-tube at $m$. Let

	\begin{equation*}\label{Tphi}
		T_m\varphi^{-1}: \TG_0\oplus \TG_1\oplus \NG_0\oplus \NG_1 \to T_{\varphi^{-1}(m)}\left( G\times_{G_m} \left(\mathfrak{m}^*\times \NG_1\right)\right).
	\end{equation*}
be the linearisation of $\varphi^{-1}$ at $m$. For $x+y\in\mathfrak{g}_{m}\oplus \mathfrak{m}$ and $z\in \mathfrak{n}$ it is given by

	\begin{equation*}
		T_m\varphi^{-1}\cdot ((x+y)_M(m)+z_M(m)+w+\nu)=T_{(e,0,0)}\pi\cdot (x+y+z,f(w),\nu)
	\end{equation*}
where $\pi:G\times \mathfrak{m}^*\times \NG_1\to G\times_{G_m}\left(\mathfrak{m}^*\times \NG_1\right)$ is the orbit map. By definition, the subspace $\ker\left(D\Phi_H(m)\right)$ consists of the elements $$((x+y)_M(m)+z_M(m)+w+\nu)\in \TG_0\oplus \TG_1\oplus \NG_0\oplus \NG_1$$ satisfying $D(\restr{\Phi_H}{U}\circ \varphi\circ \pi)(e,0,0)\cdot (x+y+z,f(w),\nu)=0.$ Equivalently 

	\begin{equation}\label{KerJH: eq1}
		\partialt \restr{\Phi_H}{U}\circ \varphi\left([\exp({t(x+y+z)}),tf(w),t\nu]\right)=0
	\end{equation}
Since $G$ is connected, we can use Theorem \ref{MGS} to write $\restr{\Phi_H}{U}\circ \varphi=i^*_{\mathfrak{h}}\circ \widetilde{\Phi}_G$ with $\widetilde{\Phi}_G$ as in \eqref{Jyr}. Equation \eqref{KerJH: eq1} becomes

	\begin{equation*}
		\partialt i^*_{\mathfrak{h}}\left(\Ad^*_{\exp(-t(x+y+z))}\left(\mu+tf(w)+\Phi_{\NG_1}(t\nu)\right)\right)=i^*_{\mathfrak{h}}\left(-\ad^*_z\mu+f(w)\right)=0.
	\end{equation*}
Then $-\ad^*_z\mu+f(w)\in\mathfrak{h}^{\circ}$ since the kernel of $i^*_{\mathfrak{h}}$ is equal to $\mathfrak{h}^{\circ}$. We conclude that
$\ker\left(\D\Phi_H(m)\right)= \ker\left(\D\Phi_G(m)\right)\oplus\mathcal{M}$ where

	\begin{equation}\label{M}
		\mathcal{M}:=\lbrace z_M(m)+w\in \TG_1\oplus \NG_0\mid -\ad^*_z\mu+f(w)\in \mathfrak{h}^{\circ}\rbrace.
	\end{equation}

It remains to show that $\mathcal{M}$ is isomorphic to $\mathfrak{q}\cdot m \oplus\mathfrak{b}^*$. By construction $$\TG_1=\mathfrak{n}\cdot m=(\mathfrak{q}\oplus (\mathfrak{h}^{\perp_{\mu}})^{\perp_{\mathfrak{g}}})\cdot m$$ and $\NG_0$ is isomorphic to $\mathfrak{m}^*= \mH^*\oplus \mathfrak{b}^*$. An element $z_M(m)+w\in\mathcal{M}$ can thus be written uniquely as $u_M(m)+v_M(m)+w$ for some unique elements $u\in \mathfrak{q}, v\in (\mathfrak{h}^{\perp_{\mu}})^{\perp_{\mathfrak{g}}}$ and $w\in \NG_0$. In addition, we set $f(w)=\pi+\beta$ for $\pi\in \mH^*$ and $\beta\in\mathfrak{b}^*$. By definition of $\mathcal{M}$ the following relation holds:

	\begin{equation}\label{dec}
		\langle -\ad^*_{u+v}\mu+\pi+\beta,\eta\rangle=0\quad\mbox{ for every }\quad\eta\in\mathfrak{h}.
	\end{equation}
From the decomposition $$\mathfrak{g}^*_{\mu}=\mathfrak{h}^*_{\mu}\oplus (\mathfrak{h}_m^{\perp_{\mathfrak{g}_m}})^*\oplus \mathfrak{b}^*,$$ we see that $\langle \beta, \eta\rangle=0$ for every $\eta\in \mathfrak{h}$ since $\mathfrak{g}_{\mu}\cap\mathfrak{h}=\mathfrak{h}_{\mu}$ on which $\beta$ vanishes. In addition, $\langle -\ad^*_{u+v}\mu,\eta\rangle=\langle -\ad_v^*\mu,\eta\rangle$ for every $\eta\in\mathfrak{h}$
as $u\in \mathfrak{q}\subset \mathfrak{h}^{\perp_{\mu}}$. Hence \eqref{dec} reduces to 

	\begin{equation}
		\langle -\ad_v^*\mu+\pi,\eta\rangle =0\quad\mbox{ for every }\quad\eta\in\mathfrak{h}.
	\end{equation}
In particular, if $\eta\in\mathfrak{h}_{\mu}$, we are left with $\langle \pi,\eta\rangle=0$ and thus $\pi=0$. Since $\langle -\ad_v^*\mu,\eta\rangle=0$ for every $\eta\in \mathfrak{h}$, this implies that $v\in \mathfrak{h}^{\perp_{\mu}}\cap (\mathfrak{h}^{\perp_{\mu}})^{\perp_{\mathfrak{g}}}=\lbrace 0\rbrace$. Therefore the element $z_M(m)+w$ we started with is such that $z=u\in \mathfrak{q}$ and $f(w)=\beta\in\mathfrak{b}^*$.

Conversely, it is straightforward to check from the argument above that an element $z_M(m)+w\in \mathfrak{q}\cdot m\oplus \NG_0$ such that $f(w)=\beta\in \mathfrak{b}^*$ satisfies $-\ad^*_z\mu+\beta\in\mathfrak{h}^{\circ}$. We showed that $$\mathcal{M}=\lbrace u_M(m)+w\in\mathfrak{q}\cdot m\oplus\NG_0\mid f(w)\in \mathfrak{b}^*\rbrace.$$ 
The isomorphism is $F:u_M(m)+w\in\mathcal{M}\mapsto (u_M(m),f(w))\in \mathfrak{q}\cdot m\oplus\mathfrak{b}^*$.
\end{proof}

\begin{theorem}[Compatible Symplectic Slice]\label{cor: KerJH}
		Given \eqref{Witt Artin G}, a symplectic slice $\NH_1$ at $m$ relative to the $H$-action can be chosen of the form

	\begin{equation}\label{NH1}
		\NH_1=\SS(G,H,\mu)\cdot m\oplus X_m\oplus\NG_1,
	\end{equation}
where $X_m=\mathfrak{b}\cdot m\oplus Y_m$ with $Y_m\subset \NG_0$ isomorphic to $\mathfrak{b}^*$. 
\end{theorem}

\begin{proof}
Let a Witt-Artin decomposition of $M$ as in \eqref{Witt Artin G}. Then by \eqref{gmu decomposition}

	\begin{equation}\label{2}
		\begin{array}{lcl}
				\ker\left(D\Phi_G(m)\right)
				& = &\mathfrak{g}_{\mu}\cdot m\oplus \NG_1\\ 
				& = &(\mathfrak{h}_{\mu}\oplus \mathfrak{h}_m^{\perp_{\mathfrak{g}_{m}}}\oplus \mathfrak{b})\cdot m\oplus \NG_1 \\ 
				& = &\mathfrak{h}_{\mu}\cdot m\oplus\mathfrak{b}\cdot m\oplus \NG_1.
		\end{array}
	\end{equation}
By Proposition \ref{KerJH}, there is a subspace $Y_m\subset \NG_0$ isomorphic to $\mathfrak{b}^*$ such that

	\begin{equation*}
		\begin{array}{lcll}
			\ker\left(D\Phi_H(m)\right) 
			& = & \ker\left(D\Phi_G(m)\right)\oplus \mathfrak{q}\cdot m\oplus Y_m & \\ 
			& = & \mathfrak{h}_{\mu}\cdot m\oplus\mathfrak{b}\cdot m\oplus \NG_1\oplus \mathfrak{q}\cdot m\oplus Y_m & \mbox{ from \eqref{2}}.\\
		\end{array}
	\end{equation*}		

In \eqref{gmu+Ha} and \eqref{new q} we obtained $\mathfrak{h}_{\alpha}=\mathfrak{h}_{\mu}\oplus\mathfrak{a}$ and $\mathfrak{q}=\mathfrak{a}\oplus\SS(G,H,\mu)$.
Therefore $$\mathfrak{h}_{\mu}\cdot m\oplus\mathfrak{q}\cdot m=\mathfrak{h}_{\alpha}\cdot m\oplus \SS(G,H,\mu)\cdot m.$$ Setting $X_m=\mathfrak{b}\cdot m\oplus Y_m$, we conclude that

	\begin{equation}
		\ker\left(D\Phi_H(m)\right) = \mathfrak{h}_{\alpha}\cdot m\oplus \SS(G,H,\mu)\cdot m\oplus X_m\oplus \NG_1.
	\end{equation}
A symplectic slice $\NH_1$ at $m$ for the $H$-action must satisfy $$\ker\left(D\Phi_H(m)\right)=\mathfrak{h}_{\alpha}\cdot m\oplus \NH_1.$$ Hence we choose $\NH_1=\SS(G,H,\mu)\cdot m\oplus X_m\oplus \NG_1$.
\end{proof}

\begin{lemma}\label{symplectic thing st}
		The subspace $\SS(G,H,\mu)\cdot m=\left\lbrace x_M(m)\mid x\in\SS(G,H,\mu)\right\rbrace$ is a symplectic vector subspace of $(T_mM,\omega(m))$. The restriction of $\omega(m)$ on $\mathfrak{s}(G,H,\mu)\cdot m$ coincides with the Kostant-Kirillov-Souriau symplectic form.
\end{lemma}

\begin{proof}
		Using \eqref{new q}, the complement to $\mathfrak{g}_{\mu}$ in $\mathfrak{g}$ defined in \eqref{m et n} reads

	\begin{equation}\label{stn}
		\mathfrak{n}=\underbrace{\mathfrak{a}\oplus\SS(G,H,\mu)}_{\mathfrak{q}}\oplus (\mathfrak{h}^{\perp_{\mu}})^{\perp_\mathfrak{g}}.
	\end{equation}
To show that $\SS(G,H,\mu)\cdot m$ is symplectic, we use that $\mathfrak{n}\cdot m=\left\lbrace z_M(m)\mid z\in\mathfrak{n}\right\rbrace$ is a symplectic vector subspace of $(T_mM,\omega(m))$. The restriction of $\omega(m)$ on $\mathfrak{n}\cdot m$ is non-degenerate and takes the form $$\Psi(m)(x,y)=\langle \mu,[x,y]\rangle.$$ Therefore $\omega(m)$ restricted to $\mathfrak{n}\cdot m$ coincides with the Kostant-Kirillov-Souriau symplectic form. Let us show that it is also non-degenerate when restricted to $\SS(G,H,\mu)\cdot m$. Assume $x\in\SS(G,H,\mu)$ is such that $\Psi(m)(x,y)=0$ for every $y\in\SS(G,H,\mu)$.
To show non-degeneracy we must show that $x_M(m)=0$. By \eqref{stn}, any $z\in\mathfrak{n}$ can be written uniquely as $z=u+y+v$ with $u\in \mathfrak{a},y\in\SS(G,H,\mu)$ and $v\in(\mathfrak{h}^{\perp_{\mu}})^{\perp_{\mathfrak{g}}}$. This yields

	\begin{equation}\label{st1}
		\Psi(m)(x,z)=\Psi(m)(x,u)+\Psi(m)(x,v)
	\end{equation}
as the term $\Psi(m)(x,y)$ vanishes by assumption. Note that $$\Psi(m)(x,u)=\langle \mu,[x,u]\rangle=0$$ since $x\in \mathfrak{h}^{\perp_{\mu}}$ by \eqref{Hperpmu new} and $u\in \mathfrak{a}\subset\mathfrak{h}$ by \eqref{gmu+Ha}. Moreover the last term of \eqref{st1} vanishes. To see this we construct a Witt-Artin decomposition at $m$ relative to the $H$-action:

	\begin{equation}\label{pf: witt-artin H}
		T_mM=\TH_0\oplus\TH_1\oplus\NH_0\oplus\NH_1
	\end{equation} 
with $\NH_1$ as in Theorem \ref{cor: KerJH}. Recall that $$\ker(D\Phi_H(m))=\TH_0 \oplus \NH_1.$$ Furthermore since $\ker(D\Phi_H(m))=(\mathfrak{h}\cdot m)^{\omega(m)}$, we can write

	\begin{equation}\label{pf: Hperpmu}
		\mathfrak{h}^{\perp_{\mu}}=\left\lbrace x\in\mathfrak{g}\mid x_M(m)\in\TH_0\oplus\NH_1\right\rbrace.
	\end{equation} 
There are two possibilities: 

	\begin{enumerate}[label=(\roman*)]
		\item If $v\in \mathfrak{h}$ then $v_M(m)\in \TH_1$ since $v\in(\mathfrak{h}^{\perp_{\mu}})^{\perp_{\mathfrak{g}}}$. The subspaces $\TH_1$ and $\NH_1$ are symplectically orthogonal. Hence $\Psi(m)(x,v)=0$. 
		\item Otherwise $v_M(m)\in \NH_0$. Indeed since $v\in(\mathfrak{h}^{\perp_{\mu}})^{\perp_{\mathfrak{g}}}$,  it cannot belong to $\TH_0\oplus \NH_1$ by \eqref{pf: Hperpmu}. Since $x_M(m)\in \SS(G,H,\mu)\cdot m \subset \NH_1$ and $\TH_0\oplus\NH_0$ and $\NH_1$ are symplectically orthogonal, we conclude that $\Psi(m)(x,v)=0$.
	\end{enumerate}
		
Therefore \eqref{st1} reduces to $\Psi(m)(x,z)=0$ for every $z\in \mathfrak{n}$. Since $\mathfrak{n}\cdot m$ is symplectic we get $x_M(m)=0$.
\end{proof}

\begin{theorem}\label{cor: KerH symplectic form}
		With respect to the splitting of Theorem \ref{cor: KerJH}, the symplectic form $\omega_{\NH_1}$ reads $\Psi(m)\oplus\omega_{X_m}\oplus\omega_{\NG_1}$ with $\Psi(m)$ as in Lemma \ref{symplectic thing st} and $$\omega_{X_m}\left(b_M(m)+w,b'_M(m)+w'\right)=\langle f(w'),b\rangle-\langle f(w),b'\rangle$$
for every $b,b'\in\mathfrak{b}$, $w,w'\in Y_m$, and $f$ as in Theorem \ref{Witt} \ref{f definition}.
\end{theorem}

\begin{proof}
		By Lemma \ref{symplectic thing st}, the symplectic form on $ \SS(G,H,\mu)\cdot m$ is given by the Chu map $\Psi(m)$. Denote by $\omega_{X_m}$ the restriction of $\omega(m)$ to $X_m$. It coincides with the pullback of the canonical symplectic form on $\mathfrak{b}\oplus\mathfrak{b}^*$ along the isomorphism $$b_M(m)+w\in X_m=\mathfrak{b}\cdot m\oplus Y_m\mapsto (b,f(w))\in \mathfrak{b}\oplus\mathfrak{b}^*.$$ Therefore $\omega_{X_m}\left(b_M(m)+w,b'_M(m)+w'\right)=\langle f(w'),b\rangle-\langle f(w),b'\rangle$ for all $b,b'\in\mathfrak{b}$ and $w,w'\in Y_m$. This yields the decomposition $\omega_{\NH_1}(m)=\Psi(m)\oplus\omega_{X_m}(m)\oplus\omega_{\NG_1}$ as stated.
\end{proof}

\begin{proposition}\label{mm splitting}
		With respect to the splitting of Theorem \ref{cor: KerJH}, the momentum map $\Phi_{\NH_1}:\NH_1\to \mathfrak{h}_m^*$ associated to the linear Hamiltonian $H_m$-action on $\NH_1$ decomposes as
		
	\begin{equation*}
		\langle \Phi_{\NH_1}(\widetilde{\nu}),\eta\rangle=\frac{1}{2} \langle (\ad_x^*)^2\mu,\eta\rangle+\langle -\ad_b^*f(w),\eta\rangle+\frac{1}{2}\omega_{\NG_1}\left(\eta_{\NG_1}(\nu),\nu\right)
	\end{equation*}

for every $\eta\in\mathfrak{h}_m$, where $\widetilde{\nu}=x_M(m)+(b_M(m)+w)+\nu\in\NH_1$ with $x\in\mathfrak{s}(G,H,\mu), b\in\mathfrak{b},w\in Y_m$ and $\nu\in\NG_1$.
\end{proposition}

\begin{proof}
	By linearity of the Hamiltonian $H_m$-action on $\NH_1$, the momentum map $\Phi_{\NH_1}$ takes the form
	
	\begin{equation}\label{form}
		\langle \Phi_{\NH_1}(\widetilde{\nu}),\eta\rangle=\frac{1}{2}\omega_{\NH_1}\left(\eta_{\NH_1}(\widetilde{\nu}),\widetilde{\nu}\right)
	\end{equation}
for all $\widetilde{\nu}\in\NH_1$ and $\eta\in\mathfrak{h}_m$. With respect to the decomposition of $\NH_1$ in Theorem \ref{cor: KerJH}, we write $$\widetilde{\nu}=x_M(m)+(b_M(m)+w)+\nu\in\NH_1$$ where $x\in\mathfrak{s}(G,H,\mu), b\in\mathfrak{b},w\in Y_m$ and $\nu\in\NG_1$. For $\eta\in\mathfrak{h}_m$ we get

	\begin{eqnarray*}
		\eta_{\NH_1}(x_M(m)) &=& \partialt \mbox{exp}(t\eta)\cdot x_M(m)\\
		& = &-\pounds_{\eta_M}x_M(m) \quad(\mbox{since}\;\eta\in\mathfrak{h}_m)\\
		& = &[x_M,\eta_M](m)\\
		& = &[\eta,x]_M(m).
	\end{eqnarray*}

Similary $\eta_{\NH_1}(b_M(m))=[\eta,b]_M(m)$. 
	By Theorem \ref{cor: KerH symplectic form} we can write \eqref{form} as
	
	\begin{eqnarray*}
		\frac{1}{2}\omega_{\NH_1}\left(\eta_{\NH_1}(\widetilde{\nu}),\widetilde{\nu}\right) 
		&=& \frac{1}{2} \Psi(m)([\eta,x],x)\\
		&&+\frac{1}{2}\omega_{X_m}(m)\left([\eta,b]_M(m)+\eta_{\NH_1}(w),b_M(m)+w\right)\\
		&&+\frac{1}{2}\omega_{\NG_1}\left(\eta_{\NG_1}(\nu),\nu\right).
\end{eqnarray*}

By definition the second term of the above is $\frac{1}{2}\left(\langle f(w),[\eta,b]\rangle-\langle f(\eta_{\NH_1}(w)), b\rangle\right)$. Since the linear map $f$ is $H_m$-equivariant, $$\langle f(\eta_{\NH_1}(w)), b\rangle=\langle -\ad^*_{\eta}f(w),b\rangle=-\langle f(w),[\eta,b]\rangle.$$ 
Finally

	\begin{eqnarray*}
		\Psi(m)([\eta,x],x) &=& \langle\mu,[[\eta,x],x]\rangle\\
		&=& \langle\ad_x^*\mu,[x,\eta]\rangle\\
		&=&\langle (\ad_x^*)^2\mu,\eta\rangle.
	\end{eqnarray*} 

We thus obtain

	\begin{equation*}
		\langle \Phi_{\NH_1}(\widetilde{\nu}),\eta\rangle=\frac{1}{2} \langle (\ad_x^*)^2\mu,\eta\rangle+\langle -\ad_b^*f(w),\eta\rangle+\frac{1}{2}\omega_{\NG_1}\left(\eta_{\NG_1}(\nu),\nu\right).
	\end{equation*}
\end{proof}

\begin{example}[Abelian groups]\label{tori}
	Let $(M,\omega,G,\Phi_G)$ be a Hamiltonian $G$-manifold where $G$ is abelian and let $H$ be a subgroup of $G$. For simplicity we assume that this action is free i.e. all the stabilizers $G_m$ are trivial. If $m\in M$ has momentum $\mu=\Phi_G(m)$, then $\mathfrak{g}_{\mu}=\mathfrak{g}$ and $\mathfrak{h}_{\alpha}=\mathfrak{h}_{\mu}=\mathfrak{h}$. In particular $\mathfrak{g}_{\mu}+\mathfrak{h}_{\alpha}=\mathfrak{g}$. Since $G$ is abelian $\mathfrak{h}^{\perp_{\mu}}=\mathfrak{g}$, and thus $\mathfrak{s}(G,H,\mu)=0$ as it is the orthogonal complement of $\mathfrak{g}_{\mu}+\mathfrak{h}_{\alpha}$ in $\mathfrak{h}^{\perp_{\mu}}$. On the other hand $\mathfrak{b}=\mathfrak{h}^{\perp_{\mathfrak{g}}}$ is isomorphic to $\mathfrak{g}/\mathfrak{h}$. Theorem \ref{cor: KerJH} implies that
	
\begin{equation}
	\NH_1= \NG_1\oplus X_m\qquad\mbox{where}\qquad X_m\simeq\mathfrak{g}/\mathfrak{h}\oplus(\mathfrak{g}/\mathfrak{h})^*.
\end{equation}
\end{example}

\begin{example}\label{so(3) example}
	Let $(M,\omega,G,\Phi_G)$ where $G=\SO(3)$ is the group of rotations in $\R^3$. Assume that this action is free. Let $H=\SO(2)$ be the subgroup of rotations about the axis defined by a vector $x\in\R^3$. The Lie algebra $\mathfrak{g}$ is the space of $3\times 3$ skew-symmetric matrices. It is identified with $\R^3$ and so is its dual $\mathfrak{g}^*$ by using the standard dot product. Let $m\in M$ be a point with momentum $\Phi_G(m):=\mu\in\R^3$. Clearly $\mathfrak{g}_{\mu}:=\mbox{span}(\mu)\subset \R^3$ and $\mathfrak{h}:=\mbox{span}(x)\subset \R^3.$ Since $\ad_y^*\mu:=\mu\times y\in\R^3$ the symplectic orthogonal $\mathfrak{h}^{\perp_{\mu}}$ is the subspace of $\R^3$ defined by
	
	\begin{equation}\label{ex:hperpmu}
		\mathfrak{h}^{\perp_{\mu}}:=\lbrace y\in\R^3\mid (\mu\times x)\cdot y=0\rbrace.
	\end{equation}
There are three cases to be considered: (i) when $\mu$ and $x$ are not collinear, (ii) when they are collinear, (iii) when $\mu=0$.

\begin{enumerate}[label=(\roman*)]
	\item If $\mu$ and $x$ are not collinear then there are no elements in $H$ fixing $\mu$. Therefore $\mathfrak{h}_{\mu}=0$. As $H$ is abelian, $\mathfrak{h}_{\alpha}=\mathfrak{h}$ and thus $\mathfrak{g}_{\mu}+\mathfrak{h}_{\alpha}:=\mbox{span}(\mu,x)$. Furthermore $\mathfrak{h}^{\perp_{\mu}}:=\mbox{span}(\mu,x)$ is a $2$-dimensional plane by \eqref{ex:hperpmu}. We conclude that $\mathfrak{s}(G,H,\mu)=0$. The other subspace of interest is $\mathfrak{b}=(\mathfrak{g}_m+\mathfrak{h}_{\mu})^{\perp_{\mathfrak{g}_{\mu}}}$. In this case, as $\mathfrak{g}_m+\mathfrak{h}_{\mu}=0$, we deduce that $\mathfrak{b}=\mathfrak{g}_{\mu}:=\mbox{span}(\mu)$. By Theorem \ref{cor: KerJH} the symplectic slice for the $H$-action is given by
	
		\begin{equation}
			\NH_1= \NG_1\oplus X_m,\qquad \mbox{where}\qquad X_m\simeq\mathfrak{g}_{\mu}\oplus\mathfrak{g}_{\mu}^*.
		\end{equation}

	\item When $\mu$ and $x$ are collinear, all the elements of $H$ fix $\mu\in\R^3$. Consequently $\mathfrak{h}_{\mu}=\mathfrak{h}=\mathfrak{g}_{\mu}$ and $\mathfrak{h}^{\perp_{\mu}}=\mathfrak{g}:=\R^3.$ In this case $\mathfrak{s}(G,H,\mu)=\mathfrak{n}$ is a $2$-dimensional plane complementary to $\mathfrak{g}_{\mu}$ in $\mathfrak{g}$. However, as $\mathfrak{h}_{\mu}=\mathfrak{h}=\mathfrak{g}_{\mu}$, we find that $\mathfrak{b}=0$. Therefore,
	
		\begin{equation}\label{colinear}
			\NH_1= \NG_1\oplus \SS(G,H,\mu)\cdot m\qquad\mbox{where}\qquad \SS(G,H,\mu)\cdot m=\mathfrak{n}\cdot m.
		\end{equation}

	\item When $\mu=0$ we have $\mathfrak{g}_{\mu}=\mathfrak{g}$ and $\mathfrak{h}_{\mu}=\mathfrak{h}$. As $H$ is abelian, $\mathfrak{h}_{\alpha}=\mathfrak{h}$ and we find $\mathfrak{g}_{\mu}+\mathfrak{h}_{\alpha}=\mathfrak{g}:=\R^3$. This implies that $\mathfrak{s}(G,H,\mu)=0$. The subspace $\mathfrak{b}$ is just $\mathfrak{h}^{\perp_{\mathfrak{g}}}\simeq \mathfrak{g}/\mathfrak{h}$ which is a complement to $\mathfrak{h}$ in $\mathfrak{g}$. Therefore
	
		\begin{equation}\label{chi is zero}
			\NH_1= \NG_1\oplus X_m,\qquad\mbox{where}\qquad X_m\simeq \mathfrak{g}/\mathfrak{h}\oplus(\mathfrak{g}/\mathfrak{h})^*.
		\end{equation}

\end{enumerate}
\end{example}

\section{The case of the other subspaces}\label{section 4}

In this section we look at what happens to the other pieces of the decomposition relative to the $H$-action:

\begin{equation}\label{decomp for H}
	T_mM=\TH_0\oplus\TH_1\oplus\NH_0\oplus\NH_1.
\end{equation}

 According to \eqref{gmu decomposition} and \eqref{gmu+Ha}, we set 
 
 \begin{equation}\label{s5: TH_0 and NH_0}
 	\TH_0=\mathfrak{h}_{\alpha}\cdot m=\mH\cdot m\oplus\mathfrak{a}\cdot m\quad\mbox{ and }\quad\NH_0\simeq \mH^*\oplus \mathfrak{a}^*.
 \end{equation} 
 Recall that $\mathfrak{n}$ was defined such that $\mathfrak{g}=\mathfrak{g}_{\mu}\oplus\mathfrak{n}$. Similarly we define $\nH$ such that $\mathfrak{h}=\mathfrak{h}_{\alpha} \oplus \nH$. Hence $\TH_1=\nH\cdot m$ and the symplectic slice $\NH_1$ is as in Theorem \ref{cor: KerJH}. The same choice of splittings allows us to construct a Witt-Artin decomposition relative to the $G$-action:

\begin{equation}\label{decomp for G}
	T_mM=\TG_0\oplus\TG_1\oplus\NG_0\oplus\NG_1.
\end{equation} 
By \eqref{gmu decomposition} we get $\TG_0=\mathfrak{g}_{\mu}\cdot m=\mH\cdot m\oplus\mathfrak{b}\cdot m$. Thus $\TG_0$ is contributing to two different parts of \eqref{decomp for H}, namely $\TH_0$ by \eqref{s5: TH_0 and NH_0}, and $\NH_1$ by \eqref{NH1}. Furthermore $\NG_0$ is isomorphic to $\mathfrak{m}^*$, where $\mathfrak{m}$ is a complement of $\mathfrak{g}_m$ in $\mathfrak{g}_{\mu}$. By \eqref{m et n} there is an isomorphism $\NG_0\simeq \mH^*\oplus\mathfrak{b}^*$. Hence $\NG_0$ contributes to $\NH_0$ and $\NH_1$, by \eqref{s5: TH_0 and NH_0} and \eqref{NH1}. 

To specify $\TG_1$ note that $\mathfrak{h}^{\perp_{\mu}}\cdot m\cap\mathfrak{h}\cdot m=\mathfrak{h}_{\alpha}\cdot m$. Therefore $\nH\subset (\mathfrak{h}^{\perp_{\mu}})^{\perp_{\mathfrak{g}}}$ and let $\mathfrak{r}$ be a $G_m$-invariant complement in $(\mathfrak{h}^{\perp_{\mu}})^{\perp_{\mathfrak{g}}}$ so that $(\mathfrak{h}^{\perp_{\mu}})^{\perp_{\mathfrak{g}}}=\nH\oplus\mathfrak{r}$. By \eqref{stn} we have

\begin{equation}
	\mathfrak{n}=\mathfrak{a}\oplus\mathfrak{s}(G,H,\mu)\oplus\nH\oplus\mathfrak{r},
\end{equation}
which implies that $\TG_1=\mathfrak{n}\cdot m=\mathfrak{a}\cdot m\oplus\mathfrak{s}(G,H,\mu)\cdot m\oplus\nH\cdot m\oplus\mathfrak{r}\cdot m$. We show in Lemma \ref{last lemma} below that there is an isomorphism

\begin{equation}\label{iso a montrer}
	\TG_1=\mathfrak{n}\cdot m\simeq\mathfrak{a}\cdot m\oplus\mathfrak{s}(G,H,\mu)\cdot m\oplus\nH\cdot m\oplus\mathfrak{a}^*.
\end{equation}
Hence $\TG_1$ contributes to every subspaces appearing in \eqref{decomp for H}. Finally the symplectic slice $\NG_1$ contributes to $\NH_1$ only.

\begin{example}\label{ex: so(3)}

To illustrate the above discussion, we come back to Example \ref{so(3) example}. Note that since $H=SO(2)$ is abelian, the subspace $\TH_1$ is trivial. We have the two decompositions with respect to $G$ and $H$:

\begin{eqnarray*}
	T_mM & = &\TG_0\oplus\TG_1\oplus\NG_0\oplus \NG_1,\quad\mbox{where}\quad \TG_0\oplus\TG_1=\mathfrak{g}\cdot m= \R^3\quad\mbox{and}\quad \NG_0\simeq \TG_0^*,\\
	T_mM & = &\TH_0\oplus\NH_0\oplus \NH_1,\quad\mbox{where}\quad \TH_0=\mathfrak{h}\cdot m= \R\quad\mbox{and}\quad \NH_0\simeq \TH_0^*= \R^*.
\end{eqnarray*}

If $\mu\neq 0$, the stabilizer $G_{\mu}$ is a copy of $\SO(2)$ in $\SO(3)$. Its Lie algebra $\mathfrak{g}_{\mu}:=\mbox{span}(\mu)$ is a copy of $\R$ in $\R^3$. Therefore $\TG_0=\mathfrak{g}_{\mu}\cdot m$ is identified with a copy of $\R$ in $\R^3$, and $\TG_1= \R^2$ is a complement. Let $x\in \R^3$ such that $\mathfrak{h}:=\mbox{span}(x)$. Two cases occur:
\begin{enumerate}[label=(\roman*)]
	\item The vectors $\mu$ and $x$ are not collinear ($\mathfrak{g}_{\mu}\neq\mathfrak{h}$): In this case, $\TG_0\cap\TH_0=0$. We write $\TG_1= \R^2=\R\oplus\R$ where the first $\R$-factor corresponds to $\TH_0$ whereas the second $\R$-factor is a subspace of $\NH_1$.
	
	\item The vectors $\mu$ and $x$ are collinear ($\mathfrak{g}_{\mu}=\mathfrak{h}$): In this case, $\TG_0=\TH_0=\R$ and hence $\NG_0=\NH_0\simeq \R^*$.
	By \eqref{colinear} $\TG_1=\mathfrak{n}\cdot m$ is a subspace of $\NH_1$.
\end{enumerate}

	The remaining case is when $\mu=0$. In this case $\TG_1=0$ and $\TG_0=\R^3=\R\oplus\R^2$. In the latter, the first $\R$-factor corresponds to $\TH_0$ and the $\R^2$-factor corresponds to $\mathfrak{b}=\mathfrak{h}^{\perp_{\mathfrak{g}}}$ (cf. \eqref{chi is zero}) and this copy contributes to $\NH_1$, same for its dual $(\R^2)^*$. 
\end{example}

\begin{lemma}\label{last lemma}
	The space $Z_m:=\mathfrak{a}\cdot m \oplus \mathfrak{r} \cdot m$ is a symplectic vector subspace of $(T_mM,\omega(m))$ and is isomorphic to $\mathfrak{a}\oplus\mathfrak{a}^*$. In particular, $f:\mathfrak{r}\cdot m\to\mathfrak{a}^*$ is an isomorphism (cf. Theorem \ref{Witt} \ref{f definition}).
\end{lemma}

\begin{proof}
	To show the first statement we use the same strategy as in the proof of Lemma \ref{symplectic thing st}, by using the fact that $\mathfrak{n}\cdot m$ is a symplectic vector subspace of $(T_mM,\omega(m))$ where
	
\begin{equation}\label{pf: n}
		\mathfrak{n}=\mathfrak{a}\oplus\mathfrak{r}\oplus \mathfrak{s}(G,H,\mu)\oplus \nH.
\end{equation}
Let $x_M(m)\in Z_m$ such that $\Psi(m)(x,y)=\omega(m)(x_M(m),y_M(m))=0$ for every $y_M(m)\in Z_m$. We will show that $\Psi(m)(x,z)=0$ for every $z\in\mathfrak{n}$. Pick some $z\in\mathfrak{n}$ and write it as $z=y+u+v$ where $y\in \mathfrak{a}\oplus\mathfrak{r}$, $u\in\mathfrak{s}(G,H,\mu)$ and $v\in\nH$ by \eqref{pf: n}. Then

\begin{equation*}
	\Psi(m)(x,z) = \Psi(m)(x,y)+\Psi(m)(x,u)
	+\Psi(m)(x,v) = 0.
\end{equation*}
Indeed, the first term vanishes by hypothesis. The remaining terms vanish because $Z_m\subset \TH_0\oplus\NH_0$ whereas $\nH\cdot m\subset\TH_1$ and $\mathfrak{s}(G,H,\mu)\cdot m\subset\NH_1$. Since those subspaces are symplectically orthogonal, $\Psi(m)(x,u)=0$ as well as $\Psi(m)(x,v)=0$. Using that $\omega(m)$ is non-degenerate on $\mathfrak{n}\cdot m$, we get $x_M(m)=0$.

The second statement follows if we show that $\mathfrak{a}\cdot m$ is a Lagrangian subspace of $Z_m$, that is, $\omega(m)$ vanishes identically on $\mathfrak{a}\cdot m$. Let $x,y\in\mathfrak{a}$. In particular $x\in\mathfrak{h}^{\perp_{\mu}}$ and $y\in\mathfrak{h}$ by construction of $\mathfrak{a}$ (cf. \eqref{gmu+Ha}). Therefore

\begin{equation}
	\Psi(m)(x,y)=\langle \mu,[x,y]\rangle=0,
\end{equation} 
which shows that $\mathfrak{a}\cdot m$ is a Lagrangian subspace of $Z_m$. In particular, there is an isomorphism $Z_m\simeq \mathfrak{a}\cdot m\oplus (\mathfrak{a}\cdot m)^*$. Since $\mathfrak{a}$ has trivial intersection with $\mathfrak{g}_m$, we get that $\mathfrak{a}\cdot m\simeq \mathfrak{a}$ and thus $Z_m\simeq \mathfrak{a}\oplus\mathfrak{a}^*$.
\end{proof}






\vspace{0.5cm}
\begin{minipage}[t]{7cm}
{marine.fontaine.math@gmail.com}\\
{\tt School of Mathematics \\
University of Manchester\\
M$13$ $9$PL Manchester, UK}

\end{minipage}\hfill

\end{document}